\documentclass[reqno,11pt]{amsart}
\usepackage{bbm}
\usepackage{mathrsfs}
\usepackage{amssymb}

\usepackage[usenames,dvipsnames]{xcolor}
\usepackage{amsthm}
\usepackage{amsmath}
\usepackage{amsfonts}
\usepackage{enumerate}
\usepackage{txfonts}
\usepackage{bm}

\usepackage{psfrag}

\usepackage{graphicx}

\numberwithin{equation}{section}

\newtheorem{theorem}{Theorem}[section]
\newtheorem{lemma}[theorem]{Lemma}
\newtheorem{corollary}[theorem]{Corollary}
\newtheorem{proposition}[theorem]{Proposition}

\newtheorem{remark}[theorem]{Remark}

\theoremstyle{definition}
\newtheorem{definition}[theorem]{Definition}
\newtheorem{assumption}[theorem]{Assumption}
\newtheorem{example}[theorem]{Example}

\begin{document}

\def\be{\begin{eqnarray}}
\def\ee{\end{eqnarray}}
\def\p{\partial}
\def\no{\nonumber}
\def\eps{\epsilon}
\def\de{\delta}
\def\De{\Delta}
\def\om{\omega}
\def\Om{\Omega}
\def\f{\frac}
\def\th{\theta}
\def\la{\lambda}
\def\lab{\label}
\def\b{\bigg}
\def\var{\varphi}
\def\na{\nabla}
\def\ka{\kappa}
\def\al{\alpha}
\def\La{\Lambda}
\def\ga{\gamma}
\def\Ga{\Gamma}
\def\ti{\tilde}
\def\wti{\widetilde}
\def\wh{\widehat}
\def\ol{\overline}
\def\ul{\underline}
\def\Th{\Theta}
\def\si{\sigma}
\def\Si{\Sigma}
\def\oo{\infty}
\def\q{\quad}
\def\z{\zeta}
\def\co{\coloneqq}
\def\eqq{\eqqcolon}
\def\bt{\begin{theorem}}
\def\et{\end{theorem}}
\def\bc{\begin{corollary}}
\def\ec{\end{corollary}}
\def\bl{\begin{lemma}}
\def\el{\end{lemma}}
\def\bp{\begin{proposition}}
\def\ep{\end{proposition}}
\def\br{\begin{remark}}
\def\er{\end{remark}}
\def\bd{\begin{definition}}
\def\ed{\end{definition}}
\def\bpf{\begin{proof}}
\def\epf{\end{proof}}
\def\bex{\begin{example}}
\def\eex{\end{example}}
\def\bq{\begin{question}}
\def\eq{\end{question}}
\def\bas{\begin{assumption}}
\def\eas{\end{assumption}}
\def\ber{\begin{exercise}}
\def\eer{\end{exercise}}
\def\mb{\mathbb}
\def\mbR{\mb{R}}
\def\mbZ{\mb{Z}}
\def\mc{\mathcal}
\def\mcS{\mc{S}}
\def\ms{\mathscr}
\def\lan{\langle}
\def\ran{\rangle}
\def\lb{\llbracket}
\def\rb{\rrbracket}

\title[Smooth helically symmetric transonic flows]{Smooth helically symmetric transonic flows with nonzero vorticity in a concentric cylinder}

\author{Yi KE}
\address{School of mathematics and statistics, Wuhan University, Wuhan, Hubei Province, China, 430072.}
\email{keycloud@whu.edu.cn}

\author{Shangkun WENG}
\address{School of mathematics and statistics, Wuhan University, Wuhan, Hubei Province, China, 430072.}
\email{skweng@whu.edu.cn}

\keywords{Transonic flows, helical symmetry, step, deformation-curl decomposition, elliptic-hyperbolic mixed.}
\subjclass[2010]{35M12,76H05,76N10,35L67.}
\date{}
\maketitle

%\title{Smooth helically symmetric transonic flows with nonzero vorticity in a concentric cylinder}
%\author{Yi Ke\thanks{School of Mathematics and Statistics, Wuhan University, Wuhan, Hubei Province, 430072, People's Republic of China. Email: keycloud@whu.edu.cn}\and Shangkun Weng\thanks{School of Mathematics and Statistics, Wuhan University, Wuhan, Hubei Province, 430072, People's Republic of China. Email: skweng@whu.edu.cn}}
%%\date{}
%\maketitle

\begin{abstract}
  This paper concerns the structural stability of smooth cylindrical symmetric transonic flows in a concentric cylinder under helically symmetric perturbation of suitable boundary conditions. The deformation-curl decomposition developed by the second author and his collaborator is utilized to effectively decouple the elliptic-hyperbolic mixed structure in the steady compressible Euler equation. A key parameter in the helical symmetry is the step (denoted by $\sigma$), which denotes the magnitude of the translation along the symmetry axis after rotating one full turn. It is shown that the step determines the type of the first order partial differential system satisfied by the radial and vertical velocity. There exists a critical number $\sigma_{*}$ depending only on the background transonic flows, such that if $0<\sigma<\sigma_{*}$, one can prove the existence and uniqueness of smooth helically symmetric transonic flows with nonzero vorticity.
\end{abstract}

\section{Introduction and main results}

In this paper, we are interested in the structural stability of some smooth transonic spiral flow in a concentric cylinder under helical symmetric perturbations of suitable boundary conditions. The inviscid compressible spiral flows in a concentric cylinder $\Omega=\{(x_1,x_2,x_3): r_0<r=\sqrt{x_1^2+x_2^2}<r_1,x_3\in \mathbb{R} \}$ are governed by the following steady compressible Euler system
\begin{eqnarray}\label{1.1}
\begin{cases}
\p_{x_1}(\rho u_1)+\p_{x_2}(\rho u_2)+\p_{x_3}(\rho u_3)=0,\\% (conservation\ of\ mass)\\
\p_{x_1}(\rho u_1^2)+\p_{x_2}(\rho u_1 u_2)+\p_{x_3}(\rho u_1 u_3)+\p_{x_1} p=0,\\% (conservation\ of\ momentum)\\
\p_{x_1}(\rho u_1 u_2)+\p_{x_2}(\rho u_2^2)+\p_{x_3}(\rho u_2 u_3)+\p_{x_2} p=0,\\
\p_{x_1}(\rho u_1 u_3)+\p_{x_2}(\rho u_2 u_3)+\p_{x_3}(\rho u_3^2)+\p_{x_3} p=0,\\
\sum\limits_{j=1}^{3}\p_{x_j}(\rho u_j(e+\frac12|{\bf u}|^2)+u_j p)=0,\\% (conservation\ of\ energy)
\end{cases}
\end{eqnarray}
where $\bf u$$=(u_1,u_2,u_3)^t$ is the velocity, $\rho$ is the density, $p$ is the pressure, $e$ is the internal energy.
 Here we consider only the polytropic gas, therefore $p=A(S) \rho^\gamma $, $\gamma>1$, where $A (S)=a e^{\frac{S}{ c_v}}$, and $a$, $ c_v$ are positive constants, $e=\frac{p}{(\gamma-1) \rho}$. For later use, we denote the local sonic speed and the Bernoulli's function by
\begin{equation*}
c(\rho, A)=\sqrt{A\gamma} \rho^{\frac{\gamma-1}{2}},\ \ \ B=\frac{1}{2}\left|\bf u\right|^2+\frac{\gamma p}{(\gamma-1)\rho}.\\
\end{equation*}

Introduce the cylindrical coordinates $(r,\theta,z)$
 \begin{eqnarray}
r=\sqrt{x_1^2+x_2^2},\theta=\arctan \frac{x_2}{x_1},z=x_3\notag
\end{eqnarray}
and decompose the velocity as ${\bf u}= U_1 {\bf e}_r + U_2 {\bf e}_{\theta}+ U_3 {\bf e}_3$ with
\begin{eqnarray}
{\bf e}_r=\begin{pmatrix} \cos\theta\\ \sin\theta\\ 0\end{pmatrix},{\bf e}_{\theta}=\begin{pmatrix} -\sin\theta\\ \cos\theta\\ 0\end{pmatrix},{\bf e}_3=\begin{pmatrix}0\\ 0\\ 1\end{pmatrix}\notag
\end{eqnarray}
Then the previous system can be rewritten as
\begin{eqnarray}\label{1.2}
\begin{cases}
\p_{r}(\rho U_1)+\frac{1}{r} \p_{\theta}(\rho U_2)+\frac{1}{r} \rho U_1 +\p_{3}(\rho U_3)=0,\\% (conservation\ of\ mass)\\
(U_1 \p_{r}+\frac{U_2}{r} \p_{\theta}+U_3 \p_{3}) U_1+\frac{1}{\rho} \p_{r} p-\frac{U_2^2}{r}=0,\\% (conservation\ of\ momentum)\\
(U_1 \p_{r}+\frac{U_2}{r} \p_{\theta}+U_3 \p_{3}) U_2+\frac{1}{r \rho} \p_{\theta} p+\frac{U_1 U_2}{r}=0,\\
(U_1 \p_{r}+\frac{U_2}{r} \p_{\theta}+U_3 \p_{3}) U_3+\frac{1}{\rho} \p_{3} p=0,\\
(U_1 \p_{r}+\frac{U_2}{r} \p_{\theta}+U_3 \p_{3}) A=0.\\% (conservation\ of\ energy)
\end{cases}
\end{eqnarray}

The authors in \cite{WX21} had investigated the cylindrical symmetric smooth transonic spiral flows moving from the outer into the inner cylinders. They are described by smooth functions of the form ${\bf u}= \bar{U}_1(r) {\bf e}_r +\bar{U}_2(r){\bf e}_{\theta}$, $\rho(x)=\bar{\rho}(r)$, $p(x)=\bar{p}(r)$ and $A(x)=\bar{A}(r)$, solving the following system
\begin{eqnarray}\label{1.3}
\begin{cases}
\frac{d}{dr}(\bar{\rho} \bar{U}_1)+\frac{1}{r} \bar{\rho} \bar{U}_1=0, &0<r_0<r<r_1,\\
\bar{U}_1 \bar{U}_{1}'+\frac{1}{\rho}\frac{d}{dr}\bar{p}-\frac{\bar{U}_2^2}{r}=0,&0<r_0<r<r_1,\\
\bar{U}_1 \bar{U}_{2}'+\frac{\bar{U}_1 \bar{U}_2}{r}=0,&0<r_0<r<r_1,\\
\bar{U}_1 \bar{A}'=0,&0<r_0<r<r_1\\
\end{cases}
\end{eqnarray}
with the boundary conditions at the outer cylinder $r=r_1$:
\begin{equation}\label{1.4}
\rho(r_1)=\rho_0>0,U_1(r_1)=U_{10} \leq 0,U_2(r_1)=U_{20} \neq 0,A(r_1)=A_0>0.
\end{equation}
Let $B_0=\frac{1}{2}(U_{10}^{2}+U_{20}^{2})+\frac{\gamma}{\gamma-1}A_0 \rho_{0}^{\gamma-1}$ and the following proposition had been established in \cite{WXY21}.
\begin{proposition}\label{p11}
Suppose that the incoming flow is subsonic, i.e. $A_0\gamma \rho_{0}^{\gamma-1}>U_{10}^{2}+U_{20}^{2}$. Then there exist constants $0<r^{\#}<r_c<r_1$, where $r^{\#}$ depends only on $r_1$, $\gamma$ and the incoming flow at $r_1$,
\begin{equation*}
r_c=\sqrt{\frac{(\gamma+1)(\kappa_1^2+\kappa_2^2\rho_{c}^{2})}{2(\gamma-1)B_0 \rho_{c}^{2}}},\rho_c=(\frac{2(\gamma-1)B_0}{(\gamma+1)\gamma A_0})^{\frac{1}{\gamma-1}},\kappa_1=r_1 \rho_0 U_{10},\kappa_2=r_1 U_{20},
\end{equation*}
such that if $r^{\#}<r_0<r_c$, there exists a unique smooth irrotational transonic spiral flow $(\bar{U}_1,\bar{U}_2,\bar{\rho},\bar{A})$ to \eqref{1.3}-\eqref{1.4} in $[r_0,r_1]$ with all sonic points located at the cylinder $r=r_c \in (r_0,r_1)$, which moves from the outer cylinder to the inner one. Furthermore, all the sonic points are nonexceptional and noncharacteristically degenerate.
\end{proposition}
\par To be specific, if $U_{10} =0$, then $U_1\equiv0$, $A(r)\equiv A_0$, $U_2 (r)=\frac{\kappa_2}{r}$ and
\begin{equation*}
\rho_c=\bigg(\frac{\gamma-1}{A_0 \gamma}\bigg)^{\frac{1}{\gamma-1}} \bigg(B_0-\frac{\kappa_2^2}{2 r^2 }\bigg)^{\frac{1}{\gamma-1}},\enspace r^{\#}=\frac{\left| \kappa_2\right| }{\sqrt{2 B_0}}.
\end{equation*}
All the sonic points locate at the cylinder $r=r_c=\sqrt{\frac{\gamma+1}{2(\gamma-1)B_0 }}\left| \kappa_2\right|$. If $U_{10} <0$, then $U_1 (r)<0$ for any $r\in [r_0,r_1]$. Fix $r_0 \in (r^{\#},r_c)$, we call this smooth transonic spiral flow obtained in Proposition \ref{p11} on $[r_0,r_1]$ to be the background flow.

This paper will concentrate on the structural stability of the smooth transonic spiral flows under helical symmetric perturbations of suitable boundary conditions. Let us first describe the definition of helical symmetry. Denote the rotation transformation around the $x_3$ axis by
\begin{equation}\label{1.5}
R_{\theta }=
\begin{pmatrix}
\cos\theta  & \sin\theta  & 0 \\
 -\sin\theta & \cos\theta  & 0 \\
 0& 0 &  1\\
\end{pmatrix},\enspace\forall \theta \in \mathbb{R}.
\end{equation}
With this notation, we define the action of the helical group of transformations $G^{\sigma}$ on $\mathbb{R}^3$ by
\begin{equation}\label{1.6}
S_{\theta,\sigma}(x)=R_{\theta }(x)+
\begin{pmatrix}
 0\\
 0\\
\frac{\sigma }{2\pi }\theta
\end{pmatrix}=
\begin{pmatrix}
x_1\cos\theta +x_2\sin\theta  \\
 -x_1\sin\theta +x_2\cos\theta \\
x_3+\frac{\sigma }{2\pi }\theta
\end{pmatrix} ,\enspace\forall \theta \in \mathbb{R}.
\end{equation}
Then we say that the smooth function $f(x)$ is helically symmetric, if $f$ is invariant under the action of $G^{\sigma}$, i.e., $f(S_{\theta,\sigma}(x))=f(x)$ for all $\theta\in\mathbb{R}$. Similarly, we say that the smooth vector field ${\bf u}(x)$ is helical, if it is covariant with respect to the action of $G^{\sigma}$, i.e., $R_{\theta }({\bf u}(x))={\bf u}(S_{\theta,\sigma}(x))$ for all $\theta\in\mathbb{R}$.
\par It follows from the definitions of helical function and vector field that they are $\sigma$-periodic in $x_3$ variable. A domain $\Omega \subset \mathbb{R}^2 \times \mathbb{T}_{\sigma}$ is called a helical domain if $\forall x \in \Omega$, $S_{\theta,\sigma}(x) \in \Omega$ for any $\theta \in \mathbb{R}$. Here the corresponding $1$-dimensional torus can be denoted by $\mathbb{T}_{\sigma}=\mathbb{R}/\sigma \mathbb{Z}$.
\par To our purpose, we need to use another equivalent definition of the helical symmetry. Introduce new variables:
\begin{eqnarray}
\eta=\frac{\sigma}{2 \pi}\theta+z,\enspace \xi=\frac{\sigma}{2 \pi}\theta-z,\notag
\end{eqnarray}
Then any functions $f=f(r,\theta,z)$ can be expressed as a function in terms of $(r,\eta,\xi)$. We have the following proposition, which is proved in \cite{MTL90}.

\begin{proposition}
A smooth function $f=f(r,\theta,z)$ is a helical function if and only if, when expressed in the $(r,\xi,\eta)$ variables, it is independent of $\xi$, i.e. $f=f_h(r,\frac{\sigma}{2\pi}\theta+z)$, for some $f_h=f_h(r,\eta)$. Similarly, a smooth vector field ${\bf u}= U_1 {\bf e}_r + U_2 {\bf e}_{\theta}+ U_3 {\bf e}_3$ is helical if and only if there exist $V_1$,$V_2$,$V_3$ functions of $(r,\eta)$ such that $U_j(r,\theta,z)=V_j(r,\eta),j=1,2,3$, while the helical vector field ${\bf u}_h= V_1 {\bf e}_r + V_2 {\bf e}_{\theta}+ V_3 {\bf e}_3$. Furthermore, both the helical function $f_h$ and vector field ${\bf u}_h$ are periodic in $\eta$ with period $\sigma$.
\end{proposition}

Our goal is to find helically symmetric smooth solutions $({\bf u},\rho,p,A)$ to \eqref{1.2} with the form
\begin{eqnarray}\label{1.7}
{\bf u} = V_1 (r,\eta) {\bf e}_r +V_2 (r,\eta) {\bf e}_{\theta}+V_3 (r,\eta) {\bf e}_3,\enspace \rho=\rho_h (r,\eta),\enspace p=p_h (r,\eta),\enspace A=A_h (r,\eta).
\end{eqnarray}

Substituting \eqref{1.7} into \eqref{1.2}, one obtain
\begin{eqnarray}\label{1.8}
\begin{cases}
\p_{r}(\rho V_1)+\frac{\sigma}{2 \pi r} \p_{\eta}(\rho V_2)+\frac{1}{r} \rho V_1 +\p_{\eta}(\rho V_3)=0,\\% (conservation\ of\ mass)\\
(V_1 \p_{r}+(\frac{\sigma}{2 \pi r} V_2+V_3) \p_{\eta}) V_1+\frac{1}{\rho} \p_{r} p-\frac{V_2^2}{r}=0,\\% (conservation\ of\ momentum)\\
(V_1 \p_{r}+(\frac{\sigma}{2 \pi r} V_2+V_3) \p_{\eta}) V_2+\frac{\sigma}{2 \pi r \rho} \p_{\eta} p+\frac{V_1 V_2}{r}=0,\\
(V_1 \p_{r}+(\frac{\sigma}{2 \pi r} V_2+V_3) \p_{\eta}) V_3+\frac{1}{\rho} \p_{\eta} p=0,\\
(V_1 \p_{r}+(\frac{\sigma}{2 \pi r} V_2+V_3) \p_{\eta}) A=0.% (conservation\ of\ energy)\\
\end{cases}
\end{eqnarray}
To simplify the symbols, we still use $\rho,p,A$ in place of $\rho_h,p_h,A_h$. It is easy to see that the quantity $V_c$ defined by
\begin{equation}\label{1.9}
V_c= rV_2-\frac{\sigma}{2 \pi} V_3\\
\end{equation}
is conserved along the tracjectory:
\begin{equation}\label{1.10}
(V_1 \p_{r}+(\frac{\sigma}{2 \pi r} V_2+V_3) \p_{\eta}) V_c= 0.\\
\end{equation}

It is well-known that the steady Euler system \eqref{1.2} is hyperbolic in supersonic region, changes its type when the fluid moves across the sonic state and becomes elliptic-hyperbolic coupled in subsonic region. As for the system \eqref{1.8}, it turns out that the parameter $\sigma$ plays a key role in determining the types of the system \eqref{1.8}. To describe it clearly, denote the Mach number of the background flows:
\begin{equation*}
\bar{M}_1(r)=\frac{\bar{U}_1(r)}{c(\bar{\rho},\bar{A})},\enspace \bar{M}_2(r)=\frac{\bar{U}_2(r)}{c(\bar{\rho},\bar{A})},\enspace \bar{{\bf M}} (r)=(\bar{M}_1(r),\bar{M}_2(r))^t.
\end{equation*}
Define a critical number $\sigma_{*}>0$ as follow,
\begin{equation}\label{1.11}
\sigma_{*}=\underset{y_1 \in (r_c,r_1)}{min} 2 \pi y_1 \sqrt{\frac{1-\bar{M}_{1}^2(y_1)}{\left|\bar{{\bf M}}(y_1) \right|^2-1}}.\\
\end{equation}
For any $\sigma \in (0,\sigma_{*})$, we can formulate a class of boundary conditions to \eqref{1.8} so that a well-posedness Theory for \eqref{1.8} can be established. Consider the domain $\mathbb{D}_{\sigma}=(r_0,r_1)\times \mathbb{T}_{\sigma}$, where $\mathbb{T}_{\sigma}$ is 1 dimensional torus with period $\sigma$. On the outer surface $\left\{ (r_1,\eta): \eta \in \mathbb{T}_{\sigma}\right\}$:
\begin{eqnarray}\label{1.12}
\begin{cases}
V_c (r_1,\eta)=\kappa_2  +\varepsilon  q_c (\eta),\ \ \ \ V_3 (r_1,\eta)=\varepsilon  q_3 (\eta),\\
A (r_1,\eta)=\bar{A}+\varepsilon  \widetilde{A}_1 (\eta),B (r_1,\eta)=\bar{B}+\varepsilon  \widetilde{B}_1 (\eta),\\
\int_{0}^{\sigma } q_3 (\eta) \mathrm{d} \eta =0,
\end{cases}
\end{eqnarray}
with $q_c (\eta),q_3 (\eta),\widetilde{A}_1 (\eta),\widetilde{B}_1 (\eta) \in C^{2,\alpha }(\mathbb{T}_{\sigma})$ and $\varepsilon$ small enough. On the inner  surface $\left\{ (r_0,\eta): \eta \in \mathbb{T}_{\sigma}\right\}$ is
\begin{equation}\label{1.13}
V_1 (r_0,\eta)=\bar{U}_1(r_0)+\varepsilon  q_1 (\eta),
\end{equation}
with $q_1 (\eta) \in C^{2,\alpha }(\mathbb{T}_{\sigma})$.
\par Then we can derive the following theorem on the existence and uniqueness of smooth helical symmetric transonic spiral flows with small nonzero vorticity holds.
\begin{theorem}

Given a nonzero background flow with the radial velocity $\bar{U}_1\neq  0$, for any smooth $C^{2,\alpha }(\mathbb{T}_{\sigma})$ functions $\widetilde{A}_1 (\eta),\widetilde{B}_1 (\eta), q_c$ and $q_i,i=1,3$ with period $\sigma$, there exists a small $\varepsilon_0$ and constants $C$ depending only on background flow and boundary values, such that if $0<\varepsilon<\varepsilon_0$ and $0<\sigma < \sigma_{*}$ ,there exists a unique smooth transonic flow with nonzero vorticity
\begin{eqnarray}
{\bf u}= V_1(r,\eta) {\bf e}_r + V_2 (r,\eta) {\bf e}_{\theta}+ V_3 (r,\eta) {\bf e}_3,A(x)=A(r,\eta), B(x)=B(r,\eta)\notag
\end{eqnarray}
to \eqref{1.8} with \eqref{1.12} and \eqref{1.13}, and the following estimate holds
\begin{equation*}
\left\|V_1-\bar{U}_1\right\|_{C^{2,\alpha} (\overline{\mathbb{D}})} +\left\| V_2 -\frac{\kappa_2}{r}\right\|_{C^{2,\alpha} (\overline{\mathbb{D}})} +\left\|V_3\right\|_{C^{2,\alpha} (\overline{\mathbb{D}})}+ \left\| B-\bar{B}\right\|_{C^{2,\alpha} (\overline{\mathbb{D}})} + \left\| A-\bar{A}\right\|_{C^{2,\alpha} (\overline{\mathbb{D}})} \leq C \varepsilon.\quad(1.14)
\end{equation*}

\end{theorem}
\par As for helical symmetry, to the incompressible Navier-Stokes equations, Mahalov, Titi and Leibovich had already established the existence and uniqueness of global strong solution in invariant helical subspaces in \cite{MTL90}. Ettinger and Titi \cite{ET09} proved the existence and uniqueness of the weak solution of the three-dimensional unsteady Euler equations with helical symmetry in the absence of vorticity stretching. In \cite{BFNLT13}, the weak solutions of the three-dimensional incompressible flow equations with initial data admitting a one-dimensional symmetry group was investigated and the authors examined both the viscous and inviscid cases. Then Lopes Filho, Mazzucato, Niu, Nussenzveig Lopes and Titi \cite{FMNLT14} studied the limits of three-dimensional viscous and inviscid incompressible helical flows in an infinite circular pipe, with respectively no-slip and no-penetration boundary conditions, as the step approached infinity. They showed that, as the step became large, the three-dimensional helical flow approached a planar flow, which was governed by the so-called two-and-half Navier-Stokes and Euler equations, respectively. Recently, the authors \cite{KLW22} proved the existence of helical invariant weak solutions with finite Dirichlet integral to the steady Navier-Stokes equations with nonhomegeneous boundary condition on a non simply-connected helical symmetric domain, utilizing Leray's contradiction method and some important properties of helical symmetry.

The theory of transonic flows is closely connected with the studies of well-posedness theory for the mixed type PDEs, such as Tricomi's equation, Keldysh's equation, Von Karman equation and so on. Kuzmin \cite{K02} had established the well-posedness theory of a linear mixed type second order equation and applied it to the nonlinear perturbation problem of an accelerating smooth transonic irrotational basic flow in the potential and stream function place. Lately in \cite{WX13,WX16,WX19',WX21}, Wang and Xin have established the existence and uniqueness of continuous subsonic-sonic flow and regular transonic flows in the De Laval nozzles which is suitably flat at its throat, and the sonic points can locate only at the throat of the nozzle and the points on the nozzle wall with positive curvature. Weng and Xin studied smooth transonic flows with nonzero voricity in De Laval nozzles for a quasi two-dimensional steady Euler flow which is a generalization of the classical quasi one-dimensional model in \cite{WX23}. They established the structural stability of the smooth one-dimensional transonic flow with positive acceleration at the sonic point for the quasi two-dimensional steady Euler flow model under small perturbations of suitable boundary conditions, which implies the existence and uniqueness of a class of smooth transonic flows with nonzero vorticity and positive acceleration to the quasi two-dimensional model. Very recently, the authors in \cite{wz24a} proved the existence and uniqueness of smooth axisymmetric transonic irrotational flows with positive acceleration to the steady Euler equations with an external force.

In \cite{WXY21,WXY21'}, Weng, Xin and Yuan have investigated steady inviscid compressible transonic flows with radial symmetry in an annulus and given a complete classification of flow patterns in terms of boundary conditions at the inner and outer circle. Here we examine the structural stability of the background flow under suitable helical symmetry perturbations, therefore our problem simplifies to find helical symmetry transonic flows in a concentric cylinder satisfying suitable boundary conditions on the inner and outer cylinder respectively. We employ the deformation-curl decomposition introduced by Weng an Xin in \cite{WX19} to decouple the hyperbolic and elliptic modes. This decomposition has found numerous successful applications in transonic shock and smooth transonic flow problems, which can be found in recent works by Weng and Xin \cite{wx23a,WX23}. We find that the step determines the type of the first order PDE system for the radial and vertical velocity. There exists a critical number $\sigma_{*}$ depending only on the background transonic flows, such that if $0<\sigma<\sigma_{*}$, the radial and vertical velocity satisfy a first order elliptic system and one can prove the structural stability of steady transonic Euler flows under helically symmetric perturbation. The stability result for the large $\sigma$ case will be reported in a forthcoming paper.

This paper will be arranged as follows. In Section $2$, we introduce some key properties of background flows which will play a significant role in searching for the appropriate range of the period $\sigma$. In Section $3$, we consider the structural stability of the background flows within the class of helical symmetric flows. It will be shown that the quantities $V_1$ and $V_3$ satisfy a first order elliptic system when linearized around the background transonic flows. Meanwhile, $V_c$, $B$ and $A$ are conserved along the particle trajectory. The maximum principle and some suitable barrier functions are employed to derive some uniform estimates to second order elliptic equation.

\section{Some properties of the background flow}

A direct calculation shows that
\begin{eqnarray}\label{2.1}\begin{cases}
\bar{\rho}'(r)=\frac{\bar{M}_1^{2}+\bar{M}_2^{2}}{r(1-\bar{M}_1^{2})}\bar{\rho},\\
\bar{U}_1 '(r)=-\frac{1+\bar{M}_2^{2}}{r(1-\bar{M}_1^{2})}\bar{U}_1,\enspace \bar{U}_2 '(r)=-\frac{1}{r} \bar{U}_2,\\
(\bar{M}_1^{2})' (r)=\frac{\bar{M}_1^{2}}{r(\bar{M}_1^{2}-1)}(2+(\gamma-1)\bar{M}_1^{2}+(\gamma+1)\bar{M}_2^{2}),\\
(\bar{M}_2^{2})' (r)=\frac{\bar{M}_2^{2}}{r(\bar{M}_1^{2}-1)}(2+(\gamma-3)\bar{M}_1^{2}+(\gamma-1)\bar{M}_2^{2}),\\
(\left|\bar{{\bf M}} \right|^2)'(r)=\frac{\left|\bar{{\bf M}}\right|^2}{r(\bar{M}_1^{2}-1)}(2+(\gamma-1)\left|\bar{{\bf M}}\right|^2).\\
\end{cases}
\end{eqnarray}
Thus the total Mach number $\left|\bar{{\bf M}} (r)\right|$ of the background flow monotonically increases as $r$ decreases since $\bar{M}_1^2<1$ for any $r \in [r_0,r_1]$. The more concrete properties of the Mach numbers had been established in \cite{WXY21,WXY21',WZ22}.
\par For later use, we define
\begin{equation*}
\begin{split}
&\bar{A}_{11} (r)=c^{2}(\bar{\rho})-\bar{U}_1^{2}, \enspace \bar{A}_{31} (r) =\bar{A}_{13} (r)=-\frac{\sigma }{2\pi r}\bar{U}_1\bar{U}_2, \\
&\bar{A}_{33}(r)=(1+\frac{\sigma^2}{4 \pi^2 r^2})c^2(\bar{\rho})-\frac{\sigma^2 }{4 \pi^2 r^2}\bar{U}_2^{2},\\
&e_1(r)=\frac{c^{2}(\bar{\rho})}{r}+\frac{(\gamma+1)(1+\bar{M}_2^{2})}{r(1-\bar{M}_1^{2})}\bar{U}_1^2-\frac{\gamma-1}{r}\bar{U}_1^{2}>0,\enspace \forall r \in [r_0,r_1],\\
&e_3(r)=\frac{(\gamma-1)\sigma }{2\pi r^2}\frac{\left|\bar{{\bf M}} \right|^2}{1-\bar{M_1}^{2}}\bar{U}_1\bar{U}_2,\\
\end{split}
\end{equation*}
and
\begin{equation*}
\begin{split}
&f(r)=\frac{\sigma}{2 \pi}\int_{r_0}^{r} \frac{\bar{M}_1 \bar{M}_2 (\tau )}{1-\bar{M}_1^2 (\tau )} \frac{d\tau }{\tau },\\
&\bar{k}_1(r)=\frac{e_1(r)}{c^2(\bar{\rho})-\bar{U}_1^2}=\frac{1+2 \bar{M}_1^{2}+(\gamma+1)\bar{M}_1^{2}\left|\bar{{\bf M}} \right|^2 / (1-\bar{M}_1^{2})}{r(1-\bar{M}_1^{2})}>0,\\
&\bar{k}_2(r)=f''(r)+\frac{e_1(r)f'(r)+e_3(r)}{c^2(\bar{\rho})-\bar{U}_1^2},\\
&\bar{k}_{22}(r)=\frac{\bar{A}_{33}(r)+(\bar{A}_{13}+\bar{A}_{31})(r)f'(r)}{\bar{A}_{11}(r)}+(f'(r))^2.\\
\end{split}
\end{equation*}
\begin{proposition}\label{p21}
Let $(\bar{U}_1,\bar{U}_2,\bar{\rho},\bar{A})$ be the solution to the system \eqref{1.3} with \eqref{1.4}, then the following identities hold for any $r \in [r_0,r_1]$
\begin{equation*}
\begin{split}
&\bar{k}_2(r)=-\frac{\sigma}{2 \pi r^2} \frac{\bar{M}_{1}\bar{M}_{2}(2-2\bar{M}_{1}^2 +\bar{M}_{2}^2)}{(1-\bar{M}_{1}^2)^2},\\
&\bar{k}_{22}(r)=\frac{1}{1-\bar{M}_{1}^{2}}+\frac{\sigma^2}{4 \pi^2 r^2 }\frac{1-\left|\bar{{\bf M}} \right|^2}{(1-\bar{M}_{1}^{2})^2}.\\
\end{split}
\end{equation*}
\end{proposition}
\begin{proof}
Note that $f'(r)=-\frac{\bar{A}_{13}}{\bar{A}_{11}}$. Utilizing the formula \eqref{2.1}, one can calculate
\begin{eqnarray}\nonumber
&&\bar{k}_2(r)=\frac{1}{\bar{A}_{11}^2}(-e_{1}\bar{A}_{13}+e_3 \bar{A}_{11}-\bar{A}_{11}\bar{A}_{13}^{'}+\bar{A}_{13}\bar{A}_{11}^{'})\\\nonumber
&&=\frac{1}{\bar{A}_{11}^2}\bigg[\frac{\sigma}{2\pi r}\bar{U}_{1}\bar{U}_{2}\bigg(\frac{c^2(\bar{\rho})}{r}+\frac{(\gamma+1)(1+\bar{M}_2^2)}{r(1-\bar{M}_{1}^2)}\bar{U}_1^{2}-\frac{\gamma-1}{r}\bar{U}_1^{2}\bigg)\\\nonumber
&&\quad+\bar{A}_{11}\bigg(\frac{(\gamma-1)\sigma }{2\pi r^2}\frac{\left|\bar{{\bf M}} \right|^2}{1-\bar{M}_{1}^{2}}\bar{U}_{1}\bar{U}_{2}-\frac{\sigma}{2\pi r^2}\bar{U}_{1}\bar{U}_{2}+\frac{\sigma}{2\pi r}(\bar{U}_{1}^{'}\bar{U}_{2}+\bar{U}_{1}\bar{U}_{2}^{'})\bigg)\\\nonumber
&&\quad+\frac{\sigma}{2\pi r}\bar{U}_{1}\bar{U}_{2}\bigg((\gamma+1)\bar{U}_{1}\bar{U}_{1}^{'}+(\gamma-1)\bar{U}_{2}\bar{U}_{2}^{'}\bigg)\bigg]\\\nonumber
&&=-\frac{\sigma}{2 \pi r^2} \frac{\bar{U}_{1}\bar{U}_{2}}{\bar{A}_{11}^2}(2c^2(\bar{\rho})-2\bar{U}_{1}^{2}+\bar{U}_{2}^{2})=-\frac{\sigma}{2 \pi r^2} \frac{\bar{M}_{1}\bar{M}_{2}(2-2\bar{M}_{1}^2 +\bar{M}_{2}^2)}{(1-\bar{M}_{1}^2)^2}.
\end{eqnarray}
On the other hand, there holds
\begin{eqnarray}\nonumber
&&\bar{k}_{22}(r)=\frac{1}{\bar{A}_{11}^2}(\bar{A}_{11}\bar{A}_{33}-\bar{A}_{13}^{2})\\\nonumber
&&=\frac{1}{\bar{A}_{11}^2} \bigg[ \bigg( (1+\frac{\sigma^2}{4 \pi^2 r^2})c^2(\bar{\rho})-\frac{\sigma^2 }{4 \pi^2 r^2}\bar{U}_{2}^{2}\bigg)(c^{2}(\bar{\rho})-\bar{U}_{1}^{2})-\frac{\sigma^2 }{4 \pi^2 r^2}\bar{U}_{1}^{2} \bar{U}_{2}^{2} \bigg]\\\nonumber
&&=\frac{1}{\bar{A}_{11}^2} \bigg( c^2(\bar{\rho})(c^{2}(\bar{\rho})-\bar{U}_{1}^{2})+\frac{\sigma^2 }{4 \pi^2 r^2}c^2(\bar{\rho})(c^2(\bar{\rho})-\bar{U}_{1}^{2}-\bar{U}_{2}^{2}) \bigg)\\\nonumber
&&=\frac{1}{1-\bar{M}_{1}^{2}}+\frac{\sigma^2}{4 \pi^2 r^2 }\frac{1-\left|\bar{{\bf M}} \right|^2}{(1-\bar{M}_{1}^{2})^2}.
\end{eqnarray}
\end{proof}

\section{Smooth helically symmetric transonic flows with small nonzero vorticity}
In this section, we prove Theorem 1.3. To begin with, we borrow the idea developed in \cite{WX19,we19} to derive an equivalent system for the steady Euler system \eqref{1.8}.

With the definition \eqref{1.9}, we see
\begin{equation}\label{3.2}
V_2=\frac{1}{r}(V_c+\frac{\sigma}{2 \pi}V_3).
\end{equation}
\par First, one can identify the hyperbolic modes in the system \eqref{1.8}. By substituting the formula \eqref{3.2}, the Bernouli's quantity and the entropy are transported by the following equations
\begin{eqnarray}\label{eq314}
&&\left (V_1 \p_{r}+(\frac{\sigma}{2 \pi r^2} V_c+(1+\frac{\sigma^2}{4 \pi^2 r^2})V_3) \p_{\eta}\right )B=0,\\\label{eq315}
&&\left (V_1 \p_{r}+(\frac{\sigma}{2 \pi r^2} V_c+(1+\frac{\sigma^2}{4 \pi^2 r^2})V_3) \p_{\eta}\right )A=0.
\end{eqnarray}
It follows from the third and fourth equation in \eqref{1.8} that
\begin{equation}\label{eq312}
\left (V_1 \p_{r}+(\frac{\sigma}{2 \pi r^2} V_c+(1+\frac{\sigma^2}{4 \pi^2 r^2})V_3) \p_{\eta}\right )V_c=0.
\end{equation}
\par Next, we derive the equations for $V_1$ and $V_3$. It follows from the fourth equation in \eqref{1.8} that
\begin{equation*}
  (V_1 \p_{r}+(\frac{\sigma}{2 \pi r} V_2+V_3) \p_{\eta}) V_3 =-(\p_{\eta}B-V_1\p_{\eta}V_1-V_2\p_{\eta}V_2-V_3\p_{\eta}V_3-\frac{1}{\gamma-1}\rho^{\gamma-1} \p_{\eta}A),
\end{equation*}
replacing $V_2$ by \eqref{3.2} yields
\begin{equation}\label{eq313}
V_1(\p_{r}V_3-\p_{\eta}V_1)=-\p_{\eta}B+\frac{1}{r^2}(V_c+\frac{\sigma}{2 \pi} V_3)\p_{\eta}V_c+\frac{1}{\gamma-1}\rho^{\gamma-1}\p_{\eta}A.
\end{equation}
We have the local sonic speed
\begin{equation}\label{3.3}
c^2 (B, {\bf V} )=(\gamma-1)\left [  B-\frac{1}{2}\left ( V_1^2+\frac{1}{r^2}V_c^2+\frac{\sigma}{\pi r^2}V_c V_3+(1+\frac{\sigma^2}{4 \pi^2 r^2})V_3^2\right )\right ],
\end{equation}
and
\begin{equation}\label{3.4}
\begin{aligned}
\rho={\bigg(\frac{\gamma-1}{A \gamma}\bigg)}^{\frac{1}{\gamma-1}} {\left [  B-\frac{1}{2}\left ( V_1^2+\frac{1}{r^2}V_c^2+\frac{\sigma}{\pi r^2}V_c V_3+(1+\frac{\sigma^2}{4 \pi^2 r^2})V_3^2\right )\right ]}^{\frac{1}{\gamma-1}}.
\end{aligned}
\end{equation}
Then one can obtain
\begin{equation}\label{3.5}
\begin{aligned}
 &\frac{\partial \rho }{\partial B}=\frac{\rho}{c^2(\rho)},\enspace  \frac{\partial \rho }{\partial V_1}=-\frac{\rho}{c^2(\rho)}V_1,\enspace  \frac{\partial \rho }{\partial V_3}=-\frac{\rho}{c^2(\rho)}\bigg((1+\frac{\sigma^2}{4 \pi^2 r^2})V_3+\frac{\sigma}{2 \pi r^2}V_c \bigg),\enspace \\
 &\frac{\partial \rho }{\partial V_c}=-\frac{\rho}{c^2(\rho)}(\frac{1}{r^2}V_c+\frac{\sigma}{2\pi r^2}V_3),\enspace  \frac{\partial \rho }{\partial A}=-\frac{1}{\gamma-1}\frac{\rho}{A}.
\end{aligned}
\end{equation}
Therefore by \eqref{3.2},\eqref{3.4},\eqref{3.5}, the continuity equation in \eqref{1.8} can be transformed as
\begin{eqnarray}\nonumber
&&(c^2 (B, {\bf V} )-V_1^2) \p_{r}V_1-\frac{\sigma}{2 \pi r^2}V_1 V_c \p_{\eta}V_1-\frac{\sigma}{2 \pi r^2} V_1V_c \p_{r}V_3+\bigg((1+\frac{\sigma^2}{4 \pi^2 r^2})c^2 (B,{\bf V} )-\frac{\sigma^2}{4 \pi^2 r^4}V_c^2 \bigg)\p_{\eta}V_3\\\nonumber
&&+\frac{c^2 (B,{\bf V} )}{r}V_1
=(1+\frac{\sigma^2}{4 \pi^2 r^2})V_1V_3(\p_{\eta}V_1+\p_{r}V_3) +(1+\frac{\sigma^2}{4 \pi^2 r^2})\bigg(\frac{\sigma}{ \pi r^2}V_c+(1+\frac{\sigma^2}{4 \pi^2 r^2})V_3 \bigg)V_3\p_{\eta}V_3\\\label{eq311}
&& -\frac{\sigma}{2 \pi r^2}(c^2 (B,{\bf V} )-\frac{\sigma}{2 \pi r^2}V_3 V_c -\frac{V_c^2}{r^2})\p_{\eta}V_c - \frac{\sigma}{2 \pi r^2}V_c \p_{\eta}B +\frac{\sigma}{2 \pi r^2}\frac{c^2 (B, {\bf V} )}{(\gamma-1)A}V_c \p_{\eta}A,
\end{eqnarray}
and the left side terms are the principal part. Now we have the following lemma.
\begin{lemma}
($\textbf{Equivalence.}$) Assume that $C^1$ smooth vector functions $(\rho,{\bf u},A)$ defined on a domain $\mathbb{D}$ do not contain the vacuum (i.e. $\rho(r,\eta)>0$ in $\mathbb{D}$) and the radial velocity $V_1(r,\eta)<0$ in $\mathbb{D}$, then the following two statements are equivalent:
\begin{enumerate}[(i)]
  \item $(\rho,{\bf u},A)$ satisfy the system \eqref{1.8} in $\mathbb{D}$;
  \item $(V_1,V_3,V_c,B,A)$ satisfy the equation \eqref{eq314},\eqref{eq315},\eqref{eq312},\eqref{eq313} and \eqref{eq311} in $\mathbb{D}$.
\end{enumerate}
\end{lemma}
\begin{flushleft}
$\mathit{Proof.}$ We have proved that Statement (i) implies Statement (ii). It remained to prove the converse. The fact that the last equation in \eqref{1.8} holds is clear. Define $\rho$ by \eqref{3.4}. A direct calculation shows that \eqref{eq313} is equivalent to the fourth equation in \eqref{1.8}. Then using \eqref{eq312} and \eqref{3.2}, one can verify that the third equation in \eqref{1.8} is satisfied. The above equations together with \eqref{eq314} imply that the second equation in \eqref{1.8} holds. Finally, according to \eqref{3.4}, the continuity equation in \eqref{1.8} follows directly from \eqref{eq311}, \eqref{3.2} and \eqref{3.5}.
\end{flushleft} $\hfill\square$
\par Set
\begin{equation*}
\begin{aligned}
W_1=V_1-\bar{U}_1,W_2=V_c-\kappa_2,W_3=V_3,W_4=A-\bar{A},W_5=B-\bar{B}.
\end{aligned}
\end{equation*}
\par Now let $\overrightarrow{\mathbf{W}}=\left ( W_1,W_2,W_3,W_4,W_5 \right )$, then $\overrightarrow{\mathbf{W}}$ satisfy
\begin{eqnarray}\label{3.6}\begin{cases}
\bar{A}_{11} \p_{r}W_1+\bar{A}_{31} \p_{\eta}W_1+\bar{A}_{13} \p_{r}W_3+ \bar{A}_{33} \p_{\eta}W_3+e_1W_1+e_3W_3=G_1 (\mathbf{W}),\\
\p_{r}W_3 -\p_{\eta}W_1=G_2 (\mathbf{W}),\\
\left((W_1+\bar{U}_1) \p_{r}+(\frac{\sigma}{2 \pi r^2} (W_2+\kappa_2)+(1+\frac{\sigma^2}{4 \pi^2 r^2})W_3) \p_{\eta}\right )W_2=0,\\
\left((W_1+\bar{U}_1) \p_{r}+(\frac{\sigma}{2 \pi r^2} (W_2+\kappa_2)+(1+\frac{\sigma^2}{4 \pi^2 r^2})W_3) \p_{\eta}\right )W_4=0,\\
\left((W_1+\bar{U}_1) \p_{r}+(\frac{\sigma}{2 \pi r^2} (W_2+\kappa_2)+(1+\frac{\sigma^2}{4 \pi^2 r^2})W_3) \p_{\eta}\right )W_5=0,
\end{cases}
\end{eqnarray}
where
\begin{equation}\label{3.7}
\begin{split}
&\bar{A}_{11} (r)=c^{2}(\bar{\rho})-\bar{U}_1^{2}, \bar{A}_{31}  (r) =\bar{A}_{13}  (r) =-\frac{\sigma \kappa_2}{2\pi r^2}\bar{U}, \bar{A}_{33}(r)=(1+\frac{\sigma^2}{4 \pi^2 r^2})c^2(\bar{\rho})-\frac{\sigma^2 \kappa_2^2}{4 \pi^2 r^4},\\
&e_1(r)=\frac{c^{2}(\bar{\rho})}{r}-(\gamma+1)\bar{U}_1\bar{U}_1^{'}-\frac{\gamma-1}{r}\bar{U}_1^{2},e_3(r)=-\frac{(\gamma-1)\sigma \kappa_2}{2\pi r^2}(\frac{1}{r}\bar{U}_1+\bar{U}_1^{'}),
\end{split}
\end{equation}
and
\begin{eqnarray}\label{3.8}
\begin{cases}
G_1 (\mathbf{W})=-\frac{\sigma}{2 \pi r^2} c^2(\mathbf{W})\p_{\eta}W_2+(1+\frac{\sigma^2}{4 \pi^2 r^2})(\bar{U}_1+W_1)W_3\p_{\eta}W_1\\
\quad+\frac{\sigma}{2 \pi r^2}(1+\frac{\sigma^2}{4 \pi^2 r^2})(\kappa_2+W_2) W_3\p_{\eta}W_3
\quad+(1+\frac{\sigma^2}{4 \pi^2 r^2})W_3 \bigg( (\bar{U}_1+W_1)\p_{r}+(1+\frac{\sigma^2}{4 \pi^2 r^2})W_3\p_{\eta} \bigg) W_3\\
\quad-\frac{\sigma (\kappa_2+W_2)}{2 \pi r^2}\bigg[\p_{\eta}W_5-\frac{c^2(\mathbf{W})\p_{\eta}W_4}{(\gamma-1)(A_0+W_4)}-(1+\frac{\sigma^2}{4 \pi^2 r^2})W_3\p_{\eta}W_3-\bigg(\frac{1}{r^2}(\kappa_2+W_2)+\frac{\sigma}{2 \pi r^2}W_3\bigg)\p_{\eta}W_2 \bigg]\\
\quad+\bar{U}_1' W_1^2-\bar{U}_1'J(\mathbf{W})-\frac{1}{r}\bar{U}_1 J(\mathbf{W})+\bigg((\gamma+1)\bar{U}_1W_1+(\gamma-1)\frac{\sigma \kappa_2}{2 \pi r^2}W_3+W_1^2-J(\mathbf{W}) \bigg)\p_{r}W_1 \\
\quad+\frac{1}{r} W_1\bigg((\gamma-1)\bar{U}_1 W_1+\frac{(\gamma-1)\sigma \kappa_2}{2 \pi r^2}W_3-J(\mathbf{W}) \bigg)\\
\quad+(1+\frac{\sigma^2}{4 \pi^2 r^2})\bigg((\gamma-1)\bar{U}_1 W_1+\frac{(\gamma-1)\sigma \kappa_2}{2 \pi r^2}W_3-J(\mathbf{W}) \bigg)\p_{\eta}W_3\\
\quad+\frac{\sigma}{2 \pi r^2}\bigg(\kappa_2 W_1 \p_{r}W_3+(\bar{U}_1+W_1)W_2\p_{r}W_3 \bigg)\\
\quad+\frac{\sigma}{2 \pi r^2}W_2\bigg((\bar{U}_1+W_1)\p_{\eta}W_1+\frac{\sigma}{2 \pi r^2}(\kappa_2+W_2)\p_{\eta}W_3 \bigg)+\frac{\sigma}{2\pi r^2}\kappa_2(W_1\p_{\eta}W_1+\frac{\sigma}{2 \pi r^2}W_2\p_{\eta}W_3 ),\\
G_2 (\mathbf{W})=\frac{1}{\bar{U}_1+W_1}\bigg\{- \p_{\eta}W_5+\frac{1}{r^2}(W_2+\kappa_2+\frac{\sigma}{2\pi}W_3)\p_{\eta}W_2\\
\quad+\frac{1}{\gamma(\bar{A}+W_4)}{\bigg[\bar{B}+W_5-\frac{1}{2}\big ( V_1^2+\frac{1}{r^2}V_c^2+\frac{\sigma}{\pi r^2}V_c V_3+(1+\frac{\sigma^2}{4 \pi^2 r^2})V_3^2\big)\bigg]}\p_{\eta}W_4\bigg\},\\
J(\mathbf{W})=(\gamma-1)\bigg[W_5-\frac{\kappa_2}{r^2}W_2-\frac{1}{2}\bigg(W_1^2+\frac{W_2^2}{r^2}+(1+\frac{\sigma^2}{4 \pi^2 r^2})W_3^2+\frac{\sigma}{\pi r^2}W_2W_3\bigg)\bigg].
\end{cases}
\end{eqnarray}

Clearly, $W_2$, $W_4$ and $W_5$ satisfy hyperbolic equations, while $W_1$, $W_3$ satisfy a first order system. We set out from resolving the hyperbolic equations, then it is easy to find that $W_2$, $W_4$ and $W_5$ are of same order as $O(\varepsilon)$. Under the proper conditions, the terms involving derivatives $\nabla W_i$ for $i=1,3$ in $G_1$ contains a small factor, thus the left side of the first two equations in \eqref{3.6} are the principal part.
\par Define the solution space as:
\begin{equation*}
\Lambda=\left\{\overrightarrow{\mathbf{W}} \in C^{2,\alpha }(\overline{\mathbb{D}_{\sigma}}):\left\| \overrightarrow{\mathbf{W}} \right\|_{C^{2,\alpha }( \overline{\mathbb{D}_{\sigma}})}\leq \delta_0\right\} ,\ \ \ \ D_{\sigma}=(r_0,r_1)\times \mathbb{T}_{\sigma}
\end{equation*}
with $\delta_0 >0$ to be specified later. For any $\overline{\mathbf{W}}\in \Lambda$, we construct an operator $ \mathcal{X}:\overline{\mathbf{W}}\in \Lambda \mapsto\mathbf{W}\in \Lambda$, where $\mathbf{W}$ is derived by the following steps.

First one obtains $(W_2,W_4,W_5)$ by solving the following hyperbolic problems:
\begin{eqnarray}\label{3.9}
\begin{cases}
\left ((\bar{W}_1+\bar{U}_1)\p_{r}+(\frac{\sigma}{2 \pi r^2}(\bar{W}_2+\kappa_2)+(1+\frac{\sigma^2}{4 \pi^2 r^2})\bar{W}_3) \p_{\eta}\right )(W_2,W_4,W_5)=0,\\
( W_2,W_4,W_5)( r_1,\eta  )=( \varepsilon q_2(\eta ),\varepsilon \widetilde{A}_1(\eta ) ,\varepsilon \widetilde{B}_1(\eta ) ).
\end{cases}
\end{eqnarray}

Since $\overline{\mathbf{W}}\in \Lambda$ and $\bar{W}_1+\bar{U}_1 <0$, the above transport equations can be solved by the characteristics method. Let $(\tau,\bar{\eta}(\tau;r,\eta))$ be the solution to the following ODE:
\begin{eqnarray}\label{3.10}
\begin{cases}
\frac{\mathrm{d} \bar{\eta}(\tau;r,\eta) }{\mathrm{d} \tau}=\frac {\frac{\sigma}{2 \pi r^2}(\bar{W}_2+\kappa_2)+(1+\frac{\sigma^2}{4 \pi^2 r^2})\bar{W}_3}{\bar{W}_1+\bar{U}_1}(\tau,\bar{\eta}(\tau;r,\eta)) ,\\
\bar{\eta}(r;r,\eta)=\eta.
\end{cases}
\end{eqnarray}

Note that $\overline{\mathbf{W}}\in C^{2,\alpha }(\overline{\mathbb{D}_{\sigma}})$, it is easy to see that $\bar{\eta}(r_1;r,\eta)\in C^{2,\alpha }(\overline{\mathbb{D}_{\sigma}})$. Indeed, we have
\begin{equation}\label{3.11}
\left\|\bar{\eta}(r_1;r,\eta) \right\|_{C^{2,\alpha }\left ( \overline{\mathbb{D}_{\sigma}} \right )} \leq C_1.
\end{equation}
Then we have
\begin{eqnarray}\label{3.12}
\begin{cases}
W_2 (r,\eta)= \varepsilon q_2(\bar{\eta}(r_1;r,\eta)) ,\\
W_4 (r,\eta)=\varepsilon \widetilde{A}_1(\bar{\eta}(r_1;r,\eta)),\\
W_5 (r,\eta)=\varepsilon \widetilde{B}_1(\bar{\eta}(r_1;r,\eta))
\end{cases}
\end{eqnarray}
and
\begin{equation}\label{3.13}
\left\| ( W_2,W_4,W_5)\right\|_{C^{2,\alpha }(\overline{\mathbb{D}_{\sigma}})}\leq C_2 \varepsilon.\\
\end{equation}

Next we will solve the boundary value problem to the linear first order system to obtain $(W_1,W_3)$:
\begin{eqnarray}\label{3.14}
\begin{cases}
\bar{A}_{11} \p_{r}W_1+\bar{A}_{31} \p_{\eta}W_1+\bar{A}_{13} \p_{r}W_3+ \bar{A}_{33} \p_{\eta}W_3+e_1W_1+e_3W_3=G_1(\bar{W}_1,\bar{W}_3,W_2,W_4,W_5),\\
\p_{r}W_3 -\p_{\eta}W_1=G_2 (\bar{W}_1,\bar{W}_3,W_2,W_4,W_5),\\
W_1 (r_0,\eta)=\varepsilon q_1 (\eta),\\
W_3 (r_1,\eta)=\varepsilon q_3 (\eta).
\end{cases}
\end{eqnarray}
Here $G_1$ and $G_2$ satisfy the formula in \eqref{3.8} and the estimates \eqref{3.13}, therefore one can verify directly that
\begin{eqnarray}\label{3.15}
\begin{cases}
\left\| G_1 (\bar{W}_1,\bar{W}_3,W_2,W_4,W_5)\right\|_{C^{1,\alpha }(\overline{\mathbb{D}_{\sigma}})}\leq C_0 (\varepsilon+\delta _{0}^{2}),\\
\left\| G_2 (\bar{W}_1,\bar{W}_3,W_2,W_4,W_5)\right\|_{C^{1,\alpha }(\overline{\mathbb{D}_{\sigma}})}\leq C_0 \varepsilon .\\
\end{cases}
\end{eqnarray}
Firstly, it is easy to show that there exists a unique solution $\phi_1(r,\eta)$ to the following problem:
\begin{eqnarray}\label{3.16}
\begin{cases}
(\p_{r}^{2}+  \p_{\eta}^{2})\phi_1 =G_2 (\bar{W}_1,\bar{W}_3,W_2,W_4,W_5)\in C^{1,\alpha }(\overline{\mathbb{D}_{\sigma}}),\\
\phi_1(r_0,\eta)=\p_{r} \phi_1(r_1,\eta)=0,\ \ \forall \eta\in \mathbb{T}_{\sigma}.
%\phi_1(r_0,\eta)\equiv 0,\enspace \phi_1(r_1,0)=0.\\
\end{cases}
\end{eqnarray}
Moreover, $\phi_1(r,\eta) \in C^{3,\alpha }(\overline{\mathbb{D}_{\sigma}})$ with the property that
\begin{equation}\label{3.17}
\left\| \phi_1 \right\|_{C^{3,\alpha }(\overline{\mathbb{D}_{\sigma}})} \leq C \left\| G_2 (\bar{W}_1,\bar{W}_3,W_2,W_4,W_5)\right\|_{C^{1,\alpha }(\overline{\mathbb{D}_{\sigma}})}\leq C_0 \varepsilon ,\\
\end{equation}
Consider $ \Phi_1=W_1+\p_{\eta} \phi_1$ and $ \Phi_3=W_3-\p_{r} \phi_1$. Then by \eqref{3.14} and \eqref{3.16}:
\begin{eqnarray}\label{3.18}
\begin{cases}
\bar{A}_{11} \p_{r}\Phi_1+\bar{A}_{31} \p_{\eta}\Phi_1+\bar{A}_{13} \p_{r}\Phi_3+ \bar{A}_{33} \p_{\eta}\Phi_3+e_1\Phi_1+e_3\Phi_3=G_3(\bar{W}_1,\bar{W}_3,W_2,W_4,W_5),\\
\p_{r}\Phi_3 -\p_{\eta}\Phi_1=0,\\
\Phi_1 (r_0,\eta)=\varepsilon q_1 (\eta),\\
\Phi_3 (r_1,\eta)=\varepsilon q_3 (\eta),\\
\end{cases}
\end{eqnarray}
where
\begin{equation*}
G_3 (\mathbf{W})=G_1 (\mathbf{W})-\bigg[\bar{U}_1^2 (r)+\frac{\sigma^2}{4\pi^2r^2}\bigg(c^2(\bar{\rho})-\bar{U}_2^2 (r)\bigg)\bigg] \p_{r\eta}^{2} \phi_1-\bar{A}_{13}(\p_{r}^{2} -\p_{\eta}^{2} )\phi_1 + e_1\p_{\eta} \phi_1- e_3  \p_{r} \phi_1.\\
\end{equation*}
Now the second equation in \eqref{3.18} implies that there exists a potential function $\phi (r,\eta)$ such that $\Phi_1=\p_{r} \phi$,$\Phi_3=\p_{\eta} \phi$ and $\phi$ satisfies the following second-order equation in $\mathbb{D}$:
\begin{eqnarray}\label{3.19}
\begin{cases}
\bar{A}_{11} \p_{r}^{2}\phi+2\bar{A}_{13}\p_{r\eta}^{2} \phi+ \bar{A}_{33}\p_{\eta}^{2}\phi+e_1 (r)\p_{r}\phi+e_3 (r)\p_{\eta}\phi=G_3(\bar{W}_1,\bar{W}_3,W_2,W_4,W_5),&in\enspace \mathbb{D}_{\sigma},\\
\p_{r} \phi (r_0,\eta)=\varepsilon q_1 (\eta),&\forall \eta \in \mathbb{T}_{\sigma},\\
\phi(r_1,\eta)=\varepsilon \int_{0}^{\eta} q_3 (s)d\eta,&\forall \eta \in \mathbb{T}_{\sigma}.
\end{cases}
\end{eqnarray}

Define a new coordinate $(y_1,y_2)$ as
\begin{equation*}
y_1=r,\enspace y_2=f(r)+\eta,\\
\end{equation*}
where
\begin{equation*}
f(r)=-\int_{r_0}^{r} \frac{\bar{A}_{13}(\tau )}{\bar{A}_{11}(\tau )} \frac{d\tau }{\tau }=\frac{\sigma}{2 \pi}\int_{r_0}^{r} \frac{\bar{M}_{1}\bar{M}_{2}(\tau )}{1-\bar{M}_{1}^2 (\tau )} \frac{d\tau }{\tau }.\\
\end{equation*}
Set the function $\widehat{\phi}(y_1,y_2)=\phi(y_1,y_2-f(y_1))$. Then \eqref{3.19} can be written as
\begin{eqnarray}\label{3.22}
\begin{cases}
  \sum_{i,j=1}^{2} \bar{k}_{ij}\p_{y_i y_j}^{2}\widehat{\phi}+\sum_{i=1}^{2} \bar{k}_{i}\p_{y_i}\widehat{\phi}=\widehat{G}(\bar{W}_1,\bar{W}_3),\\
\p_{y_1} \widehat{\phi} (r_0,y_2)+ f'(r_0) \p_{y_2}  \widehat{\phi} (r_0,y_2)= \varepsilon q_1 (y_2),\\
\widehat{\phi} (r_1,y_2)=\varepsilon \int_0^{y_2-f(r_1)}q_3 (s)ds,
\end{cases}
\end{eqnarray}
where $(y_1,y_2)\in(r_0,r_1)\times \mathbb{T}_{\sigma}$ and
\begin{equation*}
\begin{split}
&\bar{k}_{11}\equiv 1,\enspace \bar{k}_{12}(y_1)=\bar{k}_{21}(y_1)=\frac{\bar{A}_{13}(y_1)+\bar{A}_{11}(y_1)f'(y_1)}{\bar{A}_{11}(y_1)}\equiv 0,\\
&\bar{k}_{22}(y_1)=\frac{\bar{A}_{33}(y_1)+(\bar{A}_{13}+\bar{A}_{31})(y_1)f'(y_1)}{\bar{A}_{11}(y_1)}+(f'(y_1))^2,\\
&\bar{k}_{1}(y_1)=\frac{e_1(y_1)}{\bar{A}_{11}(y_1)},\enspace \bar{k}_{2}(y_1)=f''(y_1)+\frac{e_1(y_1)f'(y_1)+e_2(y_1)}{\bar{A}_{11}(y_1)},\\
&\widehat{G}(\bar{W}_1,\bar{W}_3)=\frac{G_3 (\mathbf{W})}{\bar{A}_{11}(y_1)},\enspace \int_{0}^{\sigma } q_3 (y_2) \mathrm{d} y_2= 0.
%&\widehat{q}_1 (y_2)=,\enspace \widehat{q}_3 (y_2)=\varepsilon q_3 (y_2-f(r_1)).\\
\end{split}
\end{equation*}
With the Propositions \ref{p11}, \ref{p21} and the formula \eqref{2.1}, there exists a critical value $\sigma_{*}\in (0, \infty )$ defined in \eqref{1.11}, such that for any $\sigma \in [0,\sigma_{*})$, $\bar{k}_{22}(y_1)>0$, $\forall y_1 \in [r_0,r_1]$.
%\begin{equation}\label{3.21}
%\sigma_{*}=\underset{y_1 \in (r_c,r_1)}{min} 2 \pi y_1 \sqrt{\frac{1-\bar{M}_{1}^2(y_1)}{\left|\bar{{\bf M}}(y_1) \right|^2-1}}.\\
%\end{equation}
Thus the equation in \eqref{3.22} is elliptic, by Theorem 1 in \cite{L86}, there exists a unique solution $\widehat{\phi} \in C^2(\mathbb{D}_{\sigma})\cap C(\overline{\mathbb{D}_{\sigma}})$ to \eqref{3.22}. To get an $L^{\infty}$ estimate, we define a barrier function $w(y_1,y_2)=\mathfrak{M}_1(\|\widehat{G}\|_{L^{\infty }}+ \|\widehat{q}_1\|_{L^{\infty }})y_1$, where the constant $\mathfrak{M}_1$ will be specified later. Then $\widehat{\phi} -w$ satisfies
\begin{eqnarray}\label{3.23}
\begin{cases}
\p_{y_1}^{2}(\widehat{\phi}-w)+ \bar{k}_{22} \p_{y_2}^{2}(\widehat{\phi} -w)+\bar{k}_{1} \p_{y_1}(\widehat{\phi} -w)+\bar{k}_{2}\p_{y_2}(\widehat{\phi} -w)\\
\quad\quad\quad\quad\quad=\widehat{G} - \mathfrak{M}_1 \bar{k}_{1}(\|\widehat{G}\|_{L^{\infty }}+ \|\widehat{q}_1\|_{L^{\infty }}),\\
\p_{y_1} (\widehat{\phi}-w )(r_0,y_2)+ f'(r_0)\p_{y_2} (\widehat{\phi}-w )(r_0,y_2)=\varepsilon q_1 (y_2)-\mathfrak{M}_1(\|\widehat{G}\|_{L^{\infty }}+ \|\widehat{q}_1\|_{L^{\infty }}),\\
(\widehat{\phi}-w )(r_1,y_2)=\varepsilon \int_0^{y_2-f(r_1)}q_3 (s)ds-\mathfrak{M}_1 r_1(\|\widehat{G}\|_{L^{\infty }}+ \|\widehat{q}_1\|_{L^{\infty }}).
\end{cases}
\end{eqnarray}
Since $\bar{k}_{1}(y_1)>0,\forall y_1 \in[r_0,r_1]$, we can choose $\mathfrak{M}_1<0$ such that
\begin{equation*}
\widehat{G} - \mathfrak{M}_1 \bar{k}_{1}(\|\widehat{G}\|_{L^{\infty }}+ \|\widehat{q}_1\|_{L^{\infty }})>0,\forall (y_1,y_2)\in \mathbb{D}_{\sigma};\\
\end{equation*}
\begin{equation*}
\widehat{q}_1 (y_2)-\mathfrak{M}_1(\|\widehat{G}\|_{L^{\infty }}+ \|\widehat{q}_1\|_{L^{\infty }})>0,\forall y_2 \in [0,\sigma].\\
\end{equation*}
Thus, by the maximum principle, it is easy to see that $\widehat{\phi} -w$ attains its maximum only on its boundary except $\left\{ (r_0,y_2):y_2\in [0,\sigma]\right\}$ and then
\begin{equation*}
\widehat{\phi} \leq C_3 (\|\widehat{G}\|_{L^{\infty }}+ \|\widehat{q}_1\|_{L^{\infty }} +\| \widehat{q}_3\|_{L^1(\mathbb{T}_{\sigma})} ),\\
\end{equation*}
where $C_3$ is a positive constant. Analogously, one can obtain a lower bound for $\widehat{\phi}$. Thus it holds that
\begin{equation}\label{3.24}
\| \widehat{\phi}\|_{L^{\infty }(\overline{\mathbb{D}_{\sigma}})}\leq C_2( \| \widehat{G}\|_{L^{\infty }} + \sum_{i=1,3}\| \widehat{q}_i\|_{L^{\infty }}+ \| \widehat{q}_3\|_{L^1(\mathbb{T}_{\sigma})} ).\\
\end{equation}
Employing \cite[Theorem 6.6 and Theorem 6.30]{GT98}, $\widehat{\phi}$ admits the estimate
\begin{equation}\label{3.25}
\| \widehat{\phi}\|_{C^{2,\alpha}(\overline{\mathbb{D}_{\sigma}})}\leq C_4(\| \widehat{G}\|_{C^{\alpha}(\overline{\mathbb{D}_{\sigma}})} + \sum_{i=1,3}\| \widehat{q}_i\|_{C^{1,\alpha}(\mathbb{T}_{\sigma})}+ \| \widehat{\phi}\|_{L^{\infty }(\overline{\mathbb{D}_{\sigma}})} ).\\
\end{equation}
Then by estimates \eqref{3.24} and \eqref{3.25}, we obtain the $C^{2,\alpha}$ estimate in $\mathbb{D}$
\begin{equation}\label{3.26}
\| \widehat{\phi}\|_{C^{2,\alpha}(\overline{\mathbb{D}_{\sigma}})}\leq C_5(\| \widehat{G}\|_{C^{\alpha}(\overline{\mathbb{D}_{\sigma}})} + \sum_{i=1,3}\| \widehat{q}_i\|_{C^{1,\alpha}(\mathbb{T}_{\sigma})}+ \|\widehat{q}_3\|_{L^1(\mathbb{T}_{\sigma})} ).
\end{equation}
Hence by transform $\widehat{\phi}(y_1,y_2)=\phi(y_1,y_2-f(y_1))$, there exists a unique solution $\phi$ to \eqref{3.19} and the following estimate holds:
\begin{equation}\label{3.27}
\| {\phi}\|_{C^{2,\alpha}(\overline{\mathbb{D}_{\sigma}})}\leq C_6(\| G_3\|_{C^{\alpha}(\overline{\mathbb{D}_{\sigma}})} + \varepsilon \sum_{i=1,3}\| q_i\|_{C^{1,\alpha}(\mathbb{T}_{\sigma})}+ \varepsilon\| q_3\|_{L^1(\mathbb{T}_{\sigma})} ).
\end{equation}
Therefore one has shown that
\begin{equation}\label{3.28}
\| (W_1,W_3)\|_{C^{2,\alpha}(\overline{\mathbb{D}_{\sigma}})}=\| (\p_{r} \phi-\p_{\eta} \phi_1,\p_{\eta} \phi+\p_{r} \phi_1)\|_{C^{2,\alpha}(\overline{\mathbb{D}_{\sigma}})}\leq C_{*}(\varepsilon +\delta _{0}^{2}).
\end{equation}
Set $\delta_0=2 C_{*} \varepsilon$ and select a $0<\varepsilon_0<\frac{1}{4 C_{*}^{2}}$ small enough such that $C_{*}(\varepsilon +\delta _{0}^{2})\leq \delta_0$. Thus one has derived a mapping $\mathcal{X}$ from $\Lambda$ to itself. Consider $\mathbf{\mathcal{W}},\tilde{\mathbf{\mathcal{W}}}\in \Lambda$, by a similar argument to the difference of $\mathbf{\mathcal{W}}$ and $\tilde{\mathbf{\mathcal{W}}}$, then it is easy to show that $\mathcal{X}$ is a contraction mapping in a $C^{1,\alpha}$ norm. Therefore there exists a unique fixed point $\mathbf{W}\in \Lambda $ to $\mathcal{X}$, which is the desired solution. This completes the proof.

\par {\bf Acknowledgement.} Weng is partially supported by National Natural Science Foundation of China 12071359, 12221001.

{\bf No Conflict of Interest.} The authors have no conflicts to disclose.

\end{document}